\documentclass{article}

\usepackage{amsmath}
\usepackage{amsthm}
\usepackage{amsfonts}
\usepackage{amssymb}
\usepackage{graphicx}
\usepackage{enumerate}
\DeclareMathOperator*{\argmin}{argmin}

\newtheorem{thm}{Theorem}[section]

\newtheorem{lem}[thm]{Lemma}

\theoremstyle{definition}
\newtheorem{defn}[thm]{Definition}
\theoremstyle{remark}
\newtheorem{rem}[thm]{Remark}

\begin{document}

\title{On the Multiple Packing Densities of Triangles}

\author{Kirati Sriamorn\footnote{School of Mathematical Sciences, Peking University, Beijing, China}}

\maketitle

\begin{abstract}
  Given a convex disk $K$ and a positive integer $k$, let $\delta_T^k(K)$ and $\delta_L^k(K)$ denote the $k$-fold translative packing density and the $k$-fold lattice packing density of $K$, respectively. Let $T$ be a triangle. In a very recent paper \cite{sriamorn}, I proved that $\delta_L^k(T)=\frac{2k^2}{2k+1}$. In this paper, I will show that $\delta_T^k(T)=\delta_L^k(T)$.
\end{abstract}

\bigskip


\medskip


\section{Introduction}
Let $D$ be a connected subset of $\mathbb{R}^2$.
A family of bounded sets $\mathcal{F}=\{S_1,S_2,\ldots\}$ is said to be a \emph{k-fold packing} of $D$ if $\bigcup S_i\subset D$ and each point of $D$ belongs to the interiors of at most $k$ sets of the family. In particular, when all $S_i$ are translates of a fixed measurable bounded set $S$, the corresponding family is called a \emph{k-fold translative packing} of $D$ with $S$. When the translation vectors form a lattice, the corresponding family is called a \emph{k-fold lattice packing} of $D$ with $S$. Let $I=[0,1)$, and let $M(S,k,l)$ be the maximum number of bounded sets in a $k$-fold translative packing of $lI^2$ with $S$. Then, we define
$$\delta_T^k(S)=\limsup_{l\rightarrow\infty} \frac{M(S,k,l)|S|}{|lI^2|}.$$
Similarly, we can define $\delta_L^k(S)$ for the $k$-fold lattice packings.

A family of bounded sets $\mathcal{F}=\{S_1,S_2,\ldots\}$ is said to be a \emph{k-fold covering} of $D$ if each point of $D$ belongs to at least $k$ sets of the family. In particular, when all $S_i$ are translates of a fixed measurable bounded set $S$ the corresponding family is called a \emph{k-fold translative covering} of $D$ with $S$. When the translation vectors form a lattice, the corresponding family is called a \emph{k-fold lattice covering} of $D$ with $S$. Let $m(S,k,l)$ be the minimal number of translates in a $k$-fold translative covering of $lI^2$ with $S$. Then, we define
$$\vartheta_T^k(S)=\liminf_{l\rightarrow\infty} \frac{m(S,k,l)|S|}{|lI^2|}.$$
Similarly, we can define $\vartheta_L^k(S)$ for the $k$-fold lattice coverings.

We usually denote by $\delta_{T}$ and $\delta_{L}$ the $1$-fold packing densities $\delta^1_{T}$ and $\delta^1_{L}$, respectively. It is well known that $\delta_{T}(K)=\delta_{L}(K)$ holds for every convex disk $K$ \cite{rogers}.
In particular, we have $\delta_{T}(T)=\delta_{L}(T)$ for every triangle $T$. For the case that $K=C$ is a \emph{centrally symmetric} convex disk, Fejes T\'{o}th \cite{fejes} proved that $\delta(C)=\delta_{T}(C)=\delta_{L}(C)$ where $\delta(C)$ is the (congruent) packing density of $C$. In fact, we have the following statements:
\begin{itemize}
\item Let $\{C_1,C_2,\ldots,C_N\}$ be a packing of a convex hexagon $H$ (Fig. \ref{packingofH}), where $C_i$ is a convex disk. Then we can find \emph{convex polygons} $R_1,R_2,\ldots,R_N$ (Fig. \ref{polygonal}) such that
    \begin{enumerate}[(a)]
    \item $R_i\supseteq C_i$ for every $i=1,2,\ldots,N$;
    \item $\{R_1,R_2,\ldots,R_N\}$ is also a packing of $H$;
    \item the number $s_i$ of sides of $R_i$ satisfies
    \begin{equation}\label{mean_sides_inequality}
    \sum_{i=1}^{N}s_i\leq 6N.
    \end{equation}
    \end{enumerate}
\begin{figure}[!ht]
  \centering
    \includegraphics[scale=.7]{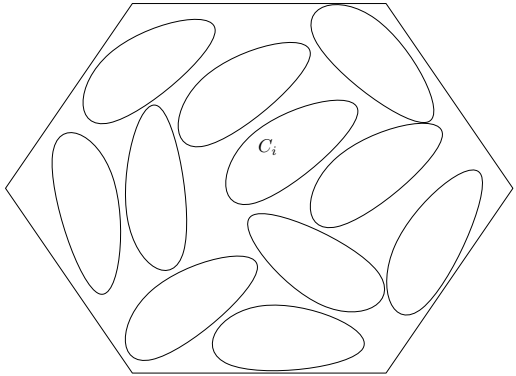}
   \caption{A packing $\{C_i\}$ of a convex hexagon $H$}\label{packingofH}
\end{figure}
\begin{figure}[!ht]
  \centering
    \includegraphics[scale=.7]{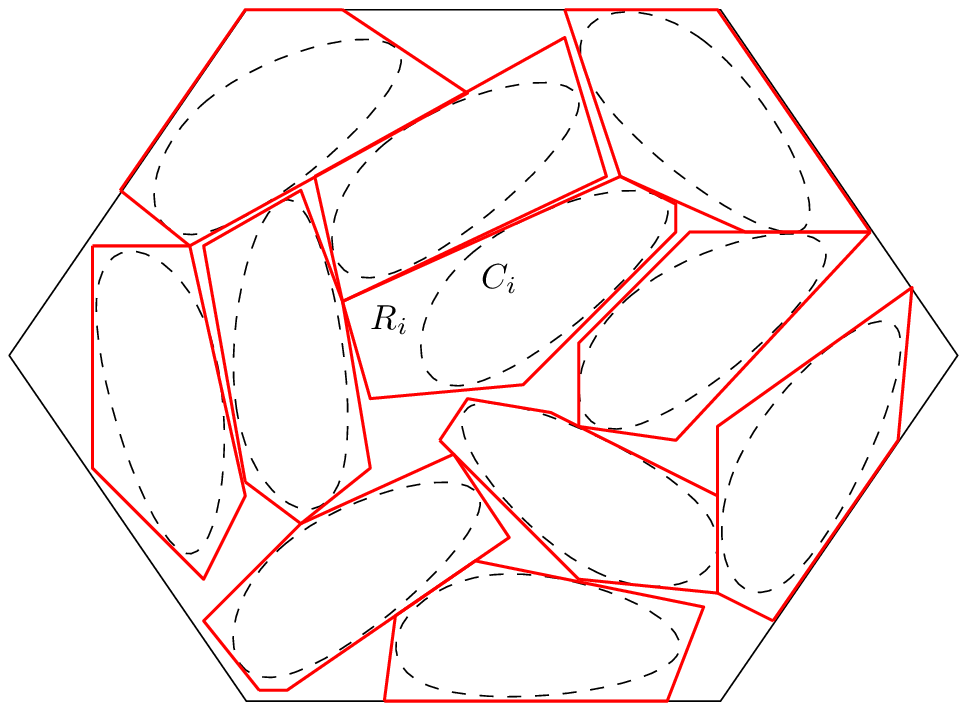}
   \caption{The polygons $R_1,R_2,\ldots,R_N$}\label{polygonal}
\end{figure}

\item (\emph{Dowker's Theorem}) Given a convex disk $C$, $n\geq 3$, let $A(n)$ denote the minimum area of an $n$-gon circumscribed about $C$. Then
    \begin{equation}\label{dowker_convexity}
    A(n)\leq \frac{A(n-1)+A(n+1)}{2}.
    \end{equation}
\item Let $C$ be a \emph{centrally symmetric} convex disk and let $n\geq 4$ be an even integer. Then one can find a convex $n$-gon $P_n$ circumscribed about $C$ with minimum area such that it is centrally symmetric and has the same center as $C$. As a consequence, we have that
    \begin{equation}\label{one_fold_lattice_density}
    \delta_L(C)\geq \frac{|C|}{|P_6|}=\frac{|C|}{A(6)},
    \end{equation}
    where $A(6)$ is the minimum area of a hexagon circumscribed about $C$.
\end{itemize}
By using these results, we can show now that $\delta(C)=\delta_{L}(C)$ where $C$ is a centrally symmetric convex disk. Let $H$ be a convex hexagon, and let $\{C_1,C_2,\ldots,C_N\}$ be a packing of $H$ with congruent copies of $C$. By (\ref{mean_sides_inequality}) and (\ref{dowker_convexity}), we have that
\begin{align*}
\frac{|C_1|+|C_2|+\cdots +|C_N|}{|H|}&=\frac{N|C|}{|H|}\\
&\leq \frac{N|C|}{|R_1|+|R_2|+\cdots+|R_N|}\\
&\leq \frac{N|C|}{A(s_1)+A(s_2)+\cdots+A(s_N)}\\
&\leq \frac{|C|}{A(6)},
\end{align*}
and hence
\begin{equation}\label{one_fold_congruent_density}
\delta(C)\leq \frac{|C|}{A(6)}.
\end{equation}
Since $\delta(C)\geq \delta_L(C)$, by (\ref{one_fold_lattice_density}) and (\ref{one_fold_congruent_density}), we obtain
\begin{equation*}
\delta(C)=\delta_L(C).
\end{equation*}

In a very recent paper, Sriamorn \cite{sriamorn} studied the $k$-fold lattice coverings and packings with  triangles $T$. He proved that
$$\delta_L^k(T)=\frac{2k^2}{2k+1},$$
and
$$\vartheta_L^k(T)=\frac{2k+1}{2}.$$
Furthermore, Sriamorn and Wetayawanich \cite{sriamorn1} showed that $\vartheta_T^k(T)=\vartheta_L^k(T)$ for every triangle $T$. In this paper, I will prove the following result:
\begin{thm} \label{main_thm}
For every triangle $T$, we have $\delta_T^k(T)=\delta_L^k(T)=\frac{2k^2}{2k+1}.$
\end{thm}

To prove this result, analogous to the proof of $\delta(C)=\delta_L(C)$ above, I will define an \emph{$r$-stair polygon} (Definition \ref{def_stair_polygon}) and use it in place of ``convex $n$-gon" above. More precisely, I will show the following statements:
\begin{itemize}
\item Let $\{T_1,T_2,\ldots,T_N\}$ be a $k$-fold \emph{translative} packing of $lI^2$ (for some positive $l$) with a triangle $T$. Then we can find \emph{stair polygons} $S_1,S_2,\ldots,S_N$ such that
    \begin{enumerate}[(a)]
    \item $S_i\supseteq T_i$ for every $i=1,2,\ldots,N$;
    \item $\{S_1,S_2,\ldots,S_N\}$ is a $k$-fold packing of $lI^2$ (Lemma \ref{S_i_exact_k_fold});
    \item we have
    \begin{equation*}
    \sum_{i=1}^{N}r_i\leq (2k-1)N,
    \end{equation*}
    where $S_i$ is an $r_i$-stair polygon (Lemma \ref{approximate_average_stair_point}).
    \end{enumerate}
\item  For $r\geq 0$, let $A^*(r)$ denote the minimum area of an $r$-stair polygon containing $T$ (see (\ref{A_star}) below). Then
    \begin{equation*}
    A^*(r)\leq \frac{A^*(r-1)+A^*(r+1)}{2}.
    \end{equation*}
\item For the case of $k$-fold lattice packings, we have
    \begin{equation*}
    \delta_L^k(T)=\frac{k|T|}{A^*(2k-1)}=\frac{2k^2}{2k+1}.
    \end{equation*}
\end{itemize}
In order to construct the desired stair polygons, I will introduce a strict partial order of triangles and a term ``press" (Section \ref{def_of_press}). Furthermore, due to technical reasons, I will introduce a concept of a \emph{normal} $k$-fold translative packing (Section \ref{def_of_normal}). An advantage of using this concept is that, by Theorem \ref{normal_packing} below, we may assume, without loss of generality, that the $k$-fold translative packing of our concern is normal, i.e., none of translates coincide. This could simplify our proof.

It is worth noting that when we talk about ($1$-fold) packings, we often refer to the concept of shadow cells \cite{pach}. I will show here a way to extend this concept to $k$-fold packings. For a nonzero vector $v$ and a point $q\in\mathbb{R}^2$, denote by $L(q,v)$ the ray parallel to $v$ and starting at $q$. Suppose that $K$ is a convex disk and $K\cap L(q,v)\neq\emptyset$, then we define
 \begin{equation*}
 \partial K(q,v)=\argmin_{p\in K\cap L(q,v)}{d(p,q)},
 \end{equation*}
 where $d(p,q)$ is the Euclidean distance between points $p$ and $q$ (Fig. \ref{raycutk}). Obviously, if $q\in K$, then $\partial K(q,v)=q$ for every nonzero vector $v$.
\begin{figure}[!ht]
  \centering
    \includegraphics[scale=1]{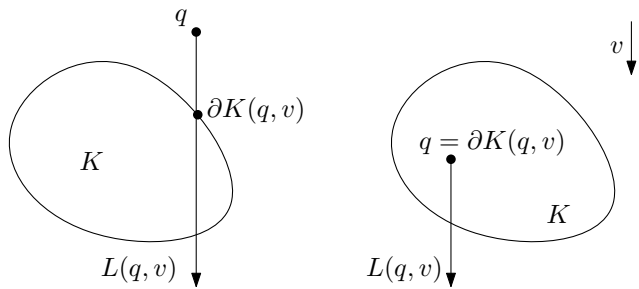}
   \caption{$\partial K(q,v)$}\label{raycutk}
\end{figure}

 \begin{defn}
(\emph{$k$-fold shadow cell}). Let a family of convex disks $\{K_1,K_2,\ldots\}$ be a $k$-fold packing of the plane and let $v$ be a nonzero vector. For every $i$, let $S_i$ be defined as the set of those points $q\in\mathbb{R}^2$ which either (i) $q\in K_i$ or (ii) $K_i\cap L(q,v)\neq\emptyset$ and there are at most $k-1$ numbers of $j$ such that $j\neq i$ and $d(\partial K_j(q,v),q)\leq d(\partial K_i(q,v),q)$. $S_i$ is called a \emph{$k$-fold shadow cell} of $K_i$ (Fig. \ref{shadowcell}).
 \end{defn}
\begin{figure}[!ht]
  \centering
    \includegraphics[scale=.61]{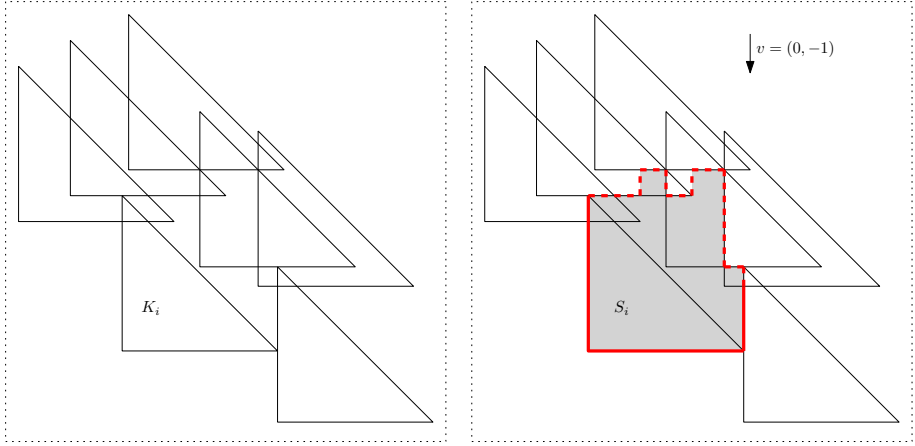}
   \caption{An example for $2$-fold shadow cell}\label{shadowcell}
\end{figure}
\begin{rem}
When $k=1$, we can give another definition of shadow cells by changing the condition $d(\partial K_j(q,v),q)\leq d(\partial K_i(q,v),q)$ in the above definition to $d(\partial K_j(q,v),q)< d(\partial K_i(q,v),q)$. Noting that the definition obtained this way will be equivalent to the definition of shadow cells described in \cite{pach}. However, we could not do the same thing for the case $k>1$, otherwise the family of shadow cells $\{S_1,S_2,\ldots\}$ might not be a $k$-fold packing of the plane. As shown in Fig. \ref{counterexample}, let $k=2$, if we use the condition $d(\partial K_j(q,v),q)< d(\partial K_i(q,v),q)$, then $q$ will lie in $S_1$, $S_2$ and $S_3$ for all $q\in D$, and hence $\{S_1,S_2,\ldots\}$ is not a $2$-fold packing.
\begin{figure}[!ht]
  \centering
    \includegraphics[scale=1]{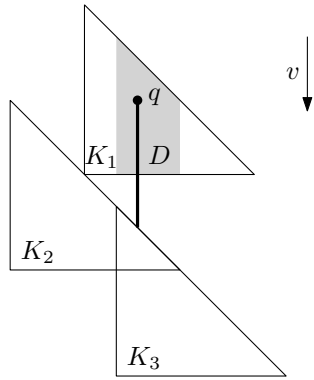}
   \caption{A counter example}\label{counterexample}
\end{figure}
\end{rem}
  Naturally, we could use the concept of $k$-fold shadow cells instead of stair polygons in our proof. However, I found that in general it is difficult to say clearly what shape the shadow cells are. Even for the case of $k$-fold translative packings with a triangle, although we might show that the shadow cells are polygons, but it is still quite hard to say how many sides they have, and hence it is not so easy to estimate their areas or to obtain the desired properties. In contrast, the shape of stair polygons is much more simple. Therefore, when we study a $k$-fold translative packing of a triangle, it seems that using the concept of stair polygons is better than using the concept of shadow cells.

\section{Normal $k$-Fold Translative Packing}\label{def_of_normal}
Let $D$ be a connected subset of $\mathbb{R}^2$ and $\mathcal{K}=\{K_1,K_2,\ldots\}$ a family of convex disks. Suppose that $\mathcal{K}$ is a $k$-fold packing of $D$. We say that $\mathcal{K}$ is \emph{normal} provided $K_i\neq K_j$ for all $i\neq j$. When $\mathcal{K}$ is normal and $K_i$ are translates of a fixed convex disk $K$, the corresponding family is called a \emph{normal k-fold translative packing} of $D$ with $K$. Let $\widetilde{M}(K,k,l)$ be the maximum number of convex disks in a normal $k$-fold translative packing of $lI^2$ with $K$. Then, we define
$$\widetilde{\delta}_T^{k}(K)=\limsup_{l\rightarrow\infty}\frac{\widetilde{M}(K,k,l)|K|}{|lI^2|}.$$

\begin{thm}\label{normal_packing}
For every convex disk $K$, we have
$$\widetilde{\delta}_T^{k}(K)=\delta_T^k(K).$$
\end{thm}
\begin{proof}
Trivially, we have that $\widetilde{\delta}_T^{k}(K)\leq\delta_T^k(K)$. Let $\{K_1,\ldots K_M\}$ be a $k$-fold translative packing of $lI^2$ with $K$. For any $K_i$, one can see that for every $0<\varepsilon<1$, there exist infinitely many points $(x,y)$ in the plane such that $K_i\supset (1-\varepsilon)K_i+(x,y)$. Hence, for every $0<\varepsilon<1$, there exist $M$ points $(x_1,y_1),\ldots,(x_M,y_M)$ in the plane such that $\{(1-\varepsilon)K_1+(x_1,y_1),\ldots,(1-\varepsilon)K_M+(x_M,y_M)\}$ is a normal $k$-fold translative packing of $lI^2$ with $(1-\varepsilon)K$. Therefore, $M\leq \widetilde{M}((1-\varepsilon)K,k,l)$. This implies that $M(K,k,l)\leq \widetilde{M}((1-\varepsilon)K,k,l)$, and hence
\begin{align*}
\delta_T^{k}(K)&=\limsup_{l\rightarrow\infty}\frac{M(K,k,l)|K|}{|lI^2|}\\
&\leq\limsup_{l\rightarrow\infty}\frac{\widetilde{M}((1-\varepsilon)K,k,l)|K|}{|lI^2|}\\
&=\frac{1}{(1-\varepsilon)^2}\widetilde{\delta}_T^{k}((1-\varepsilon)K)\\
&=\frac{1}{(1-\varepsilon)^2}\widetilde{\delta}_T^{k}(K).
\end{align*}
By letting $\varepsilon$ tend to zero, one obtains the result.
\end{proof}

\section{Some Notations}\label{def_of_press}
In this paper, we denote by $T$ the triangle with vertices $(0,0)$, $(1,0)$ and $(0,1)$.
If $T'=T+(x,y)$ where $(x,y)\in\mathbb{R}^2$, then we denote by $I^2(T')$ the square $I^2+(x,y)$, and denote by $H(T')$ the hypothenuse of $T'$.

For $(x_1,y_1),(x_2,y_2)\in\mathbb{R}^2$, we define the relation $\prec$ by $(x_1,y_1)\prec (x_2,y_2)$ if and only if either
$$x_1+y_1<x_2+y_2$$
or
$$x_1+y_1=x_2+y_2 \text{~and~} x_1<x_2.$$
One can easily show that $\prec$ is a strict partial ordering over $\mathbb{R}^2$.

Let $K$ be a nonempty bounded set. We define
$$V(K)=\{u\in\mathbb{R}^2: u\prec u' \text{~or~} u=u', \text{~for all~} u'\in K\}.$$
Denote by $v(K)$ the point $u$ in $V(K)$ such that for all $u'\in V(K)$, $u'\prec u$ or $u'=u$. For example, $v(T+(x,y))=(x,y)$ and $v(I^2+(x,y))=(x,y)$.

Suppose that $T_1$ and $T_2$ are two distinct translates of $T$ and $I^2(T_1)\cap I^2(T_2)\neq \emptyset$. We say that $T_1$ \emph{presses} $T_2$ provided $v(T_2)\prec v(T_1)$ (Fig. \ref{typeofpress}).
As immediate consequence of the definition, one can see that for every two translates $T_1,T_2$ of $T$, if $I^2(T_1)\cap I^2(T_2)\neq\emptyset$ and $T_1\neq T_2$, then either $T_1$ presses $T_2$ or $T_2$ presses $T_1$.

\begin{figure}[!ht]
  \centering
    \includegraphics[scale=.85]{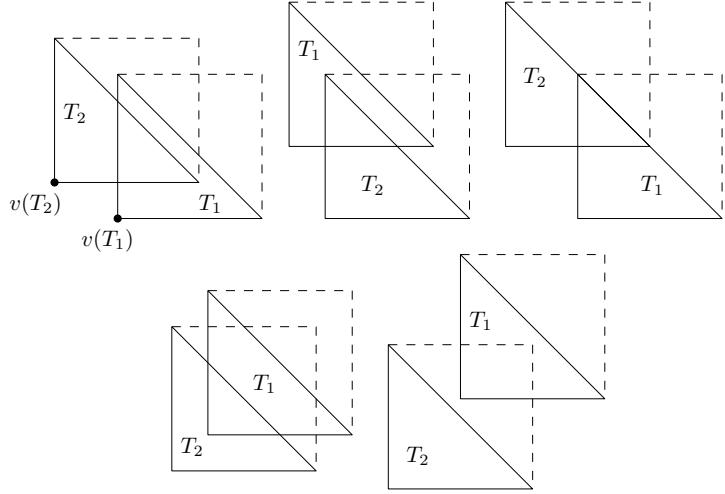}
   \caption{$T_1$ presses $T_2$}\label{typeofpress}
\end{figure}

\begin{lem}
Suppose that $T_1,~T_2$ and $T_3$ are three distinct translates of $T$ and $I^2(T_1)\cap I^2(T_3)\neq\emptyset$. If $T_1$ presses $T_2$ and $T_2$ presses $T_3$, then $T_1$ presses $T_3$.
\end{lem}
\begin{proof}
Since $T_1$ presses $T_2$ and $T_2$ presses $T_3$, we have that $v(T_2)\prec v(T_1)$ and $v(T_3)\prec v(T_2)$. Hence  $v(T_3)\prec v(T_1)$. This implies immediately that $T_1$ presses $T_3$.
\end{proof}

\begin{lem}\label{cut_by_the_others}
Suppose that $T_1,\ldots,T_n$ are $n$ distinct translates of $T$ and $I^2(T_1)\cap\cdots\cap I^2(T_n)\neq\emptyset$. Then, there exists $i\in\{1,\ldots,n\}$ such that $T_j$ presses $T_i$ for all $j\neq i$.
\end{lem}
\begin{proof}
It is easy to see that there exists $i\in\{1,\ldots,n\}$ such that $v(T_i)\prec v(T_j)$ for all $j\neq i$. We note that $I^2(T_i)\cap I^2(T_j)\neq\emptyset$. By the definition, we have that $T_j$ presses $T_i$ for all $j\neq i$.
\end{proof}

\begin{lem}\label{hypothenuse}
Let $T'$ be a translate of $T$ and $u$ be a point in $I^2(T')$. We have that $u\in T'$ if and only if $u\prec u'$ for some $u'\in H(T')$.
\end{lem}
\begin{proof}
Suppose that $T'=T+(x,y)$. If $u\in T'\cap I^2(T')$, then let $u'=(x+1,y)$. It is clear that $u\prec u'$ and $u'\in H(T')$. Conversely,
if $u\prec u'$ for some $u'\in H(T')$, then it is obvious that $u\in T'$.
\end{proof}

\begin{lem}\label{I_cut_T_must_cut}
Suppose that $T_1,\ldots,T_{n+1}$ are $n+1$ distinct translates of $T$. If $T_i$ presses $T_{n+1}$ for all $i=1,\ldots,n$, then $(I^2(T_1)\cap\cdots\cap I^2(T_n))\cap T_{n+1}\subset T_1\cap\cdots\cap T_n\cap T_{n+1}$.
\end{lem}
\begin{proof}
 Suppose that $u\in (I^2(T_1)\cap\cdots\cap I^2(T_n))\cap T_{n+1}$. Since $T_i$ presses $T_{n+1}$ and $u\in T_{n+1}$, by Lemma \ref{hypothenuse}, it is not hard to see that $u\prec u'_i$ for some $u'_i\in H(T_i)$. Again, by Lemma \ref{hypothenuse}, we have that $u\in T_i$ for all $i=1,\ldots,n$, and hence $u\in T_1\cap\cdots\cap T_n\cap T_{n+1}$.
\end{proof}

\begin{defn}\label{def_stair_polygon}
For a non-negative integer $r$, we call a planar set $S$ a \emph{half-open r-stair polygon} (Fig. \ref{stairpolygon}) if there are $x_0<x_1<\cdots< x_{r+1}$ and $y_0>y_1>\cdots > y_r>y_{r+1}$ such that
$$S=\bigcup_{i=0}^r[x_i,x_{i+1})\times[y_{r+1},y_i).$$

\begin{figure}[!ht]
  \centering
    \includegraphics[scale=.70]{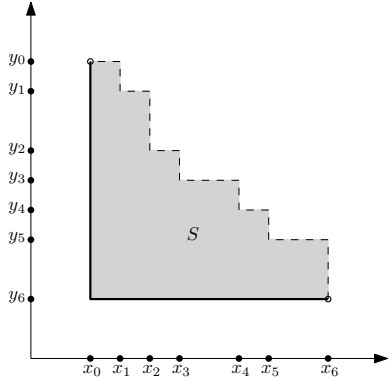}
   \caption{A half-open $5$-stair polygon}\label{stairpolygon}
\end{figure}
\end{defn}

Let $A^*(r)$ denote the minimum area of a half-open $r$-stair polygon containing $Int(T)$. Clearly, $A^*$ is a decreasing function. By elementary calculations, one can obtain
\begin{equation}\label{A_star}
A^*(r)=\frac{r+2}{2(r+1)},
\end{equation}
where $r=0,1,2,\ldots$. Let $B^*$ be the function on $[0,+\infty)$ defined by
\begin{equation}
B^*(x)=\frac{x+2}{2(x+1)},
\end{equation}
It is obvious that $B^*$ is a decreasing convex function and $B^*(r)=A^*(r)$ , for all $r=0,1,2,\ldots$. For convenience, we also denote the function $B^*$ by $A^*$. In \cite{sriamorn}, Sriamorn showed that
\begin{equation}\label{lattice_density}
\delta_L^k(T)=\frac{k|T|}{A^*(2k-1)}=\frac{2k^2}{2k+1}.
\end{equation}
\section{The Construction of Stair Polygons $S_i$}\label{key_section}
In this section, we suppose that $\mathcal{T}=\{T_1,T_2,\ldots,T_N\}$ is a \emph{normal} $k$-fold translative packing of $lI^2$ with $T$. We will use the terminologies given above to construct the desired stair polygons $S_1,S_2,\ldots,S_N$. In fact, due to technical reasons (but not essential), we will construct \emph{half-open} stair polygons instead of (closed) stair polygons. This could make it easier to prove our desired results.

Denote by $\mathcal{C}_i$ the collection of triangles $T_j$ that press $T_i$. Let
\begin{equation*}
U_i=\bigcup_{\substack{T_{i_1},\ldots,T_{i_k}\in\mathcal{C}_i \text{~are distinct}\\ I^2(T_{i_1})\cap\cdots\cap I^2(T_{i_k})\neq\emptyset}} R(v(I^2(T_{i_1})\cap\cdots\cap I^2(T_{i_k}))),
\end{equation*}
and
$$S_i=I^2(T_i)\setminus U_i,$$
where $R(x_0,y_0)$ denotes the set $\{(x,y): x\geq x_0, y\geq y_0\}$ (for example, see Fig. \ref{onefold} and Fig. \ref{twofold} ).
\begin{figure}[!ht]
  \centering
    \includegraphics[scale=.8]{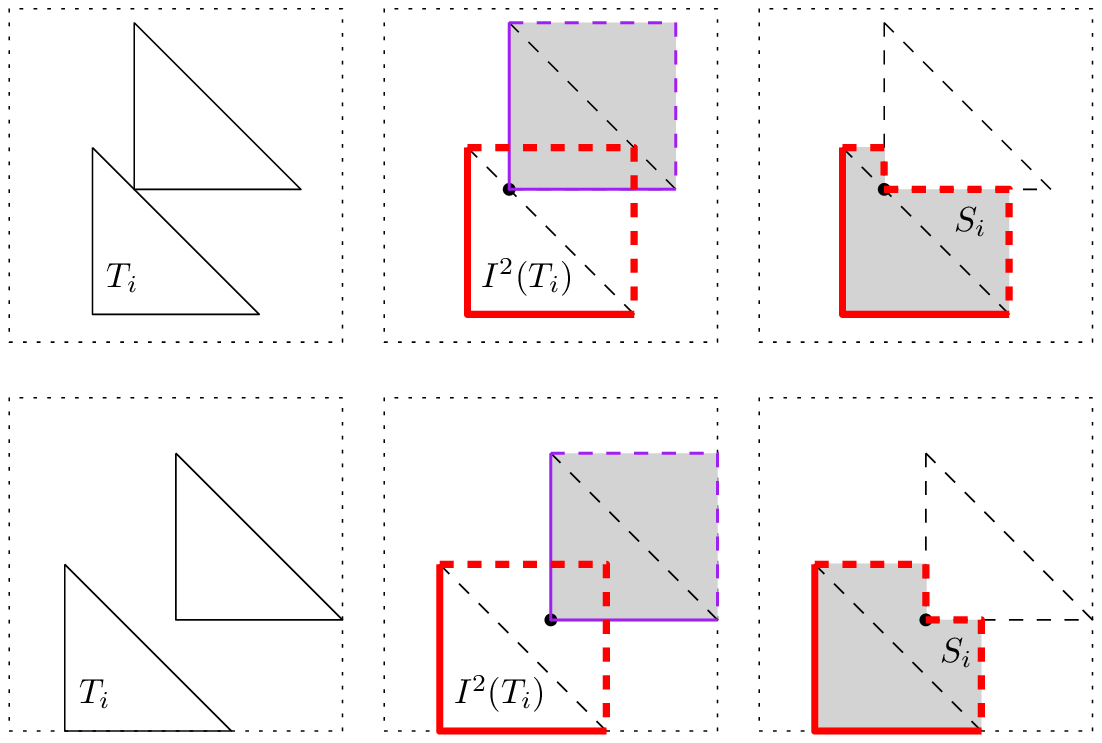}
   \caption{Two examples to illustrate the construction of stair polygons $S_i$ in a $1$-fold packing. Only triangles which press $T_i$ are shown.}\label{onefold}
\end{figure}

\begin{figure}[!ht]
  \centering
    \includegraphics[scale=.8]{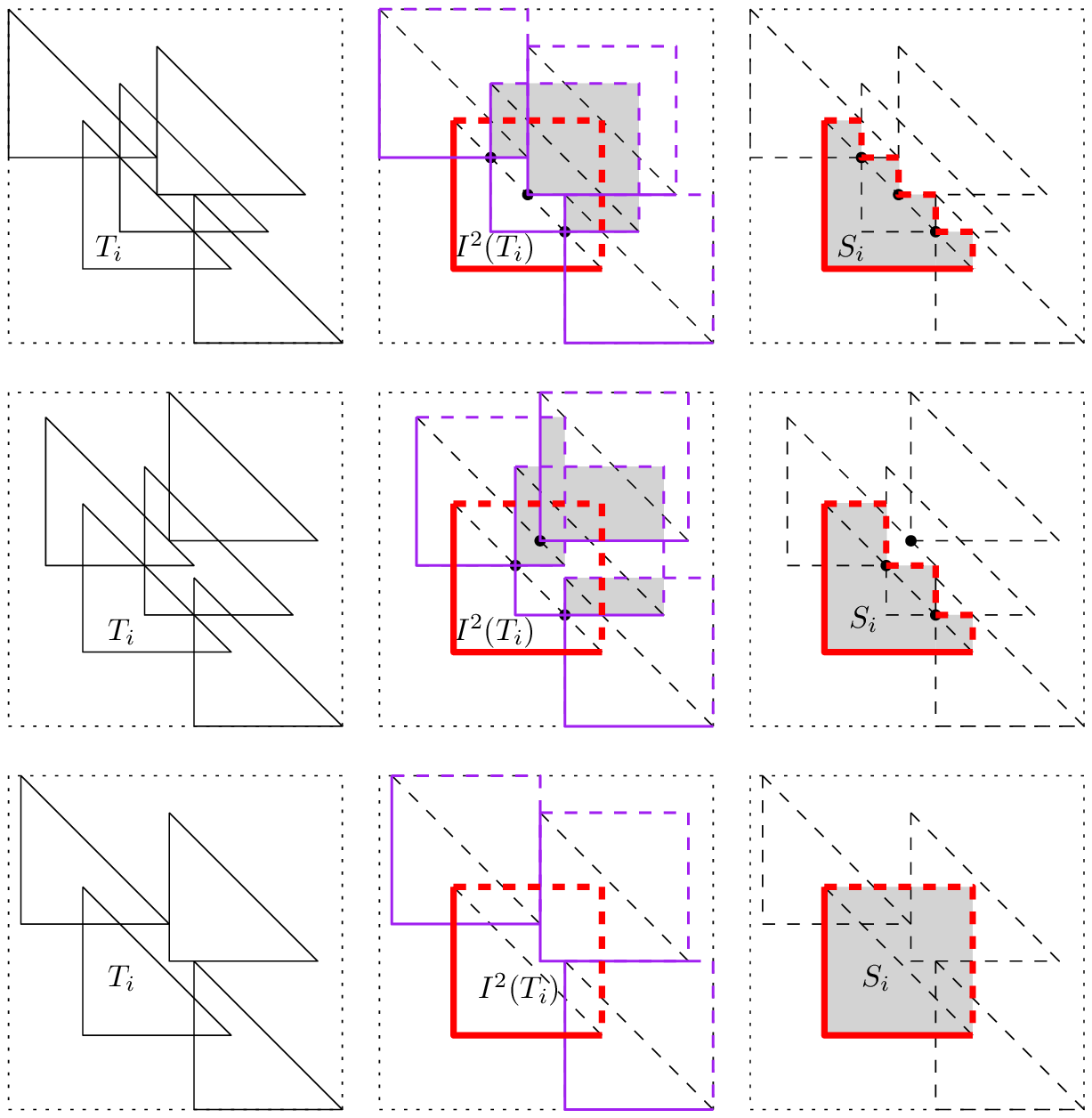}
   \caption{Three examples to illustrate the construction of stair polygons $S_i$ in a $2$-fold packing. Only triangles which press $T_i$ are shown.}\label{twofold}
\end{figure}

We have the following lemmas.
\begin{lem}\label{int_T_intersect_U}
$Int(T_i)\cap U_i=\emptyset$.
\end{lem}
\begin{proof}
Assume that $Int(T_i)\cap U_i\neq\emptyset$. By the definition of $U_i$, it can be deduced that there exist $T_{i_1},\ldots,T_{i_k}\in\mathcal{C}_i$ such that $Int(T_i)\cap(I^2(T_{i_1})\cap\cdots\cap I^2(T_{i_k}))\neq \emptyset$. By Lemma \ref{I_cut_T_must_cut}, we know that $Int(T_i)\cap(I^2(T_{i_1})\cap\cdots\cap I^2(T_{i_k}))\subset Int(T_i)\cap T_{i_1}\cap\cdots\cap T_{i_k}$, and hence $Int(T_i)\cap T_{i_1}\cap\cdots\cap T_{i_k}\neq \emptyset$. Since $v(T_i)\prec v(T_{i_j})$ for all $j=1,\ldots,k$, one can see that $H(T_i\cap T_{i_1}\cap\cdots\cap T_{i_k})\subset H(T_i)$. From $Int(T_i)\cap T_{i_1}\cap\cdots\cap T_{i_k}\neq \emptyset$, we know that $Int(T_i)\cap Int(T_{i_1})\cap\cdots\cap Int(T_{i_k})\neq \emptyset$. This is impossible, since $\mathcal{T}$ is a $k$-fold packing of $\mathbb{R}^2$.
\end{proof}

\begin{lem}\label{S_i_stair_polygon}
$S_i$ is a half-open stair polygon containing $Int(T_i)$.
\end{lem}
\begin{proof}
 We note that $\mathcal{C}_i$ is finite, hence it is obvious that $S_i$ is a half-open stair polygon. By Lemma \ref{int_T_intersect_U}, we have $Int(T_i)\subset I^2(T_i)\setminus U_i=S_i$.
\end{proof}

We may assume, without loss of generality, that $S_i$ is a half-open $r_i$-stair polygon and
$$S_i=\bigcup_{j=0}^{r_i}[x_j^{(i)},x_{j+1}^{(i)})\times[y_{r_i+1}^{(i)},y_j^{(i)}),$$
where $x_0^{(i)}<x_1^{(i)}<\cdots<x_{r_i+1}^{(i)}$ and $y_0^{(i)}>y_1^{(i)}>\cdots>y_{r_i+1}^{(i)}$ (Fig. \ref{Si}). Let
$$Z(S_i)=\{(x_j^{(i)},y_j^{(i)}): j=1,\ldots,r_i\}.$$

\begin{figure}[!ht]
  \centering
    \includegraphics[scale=.90]{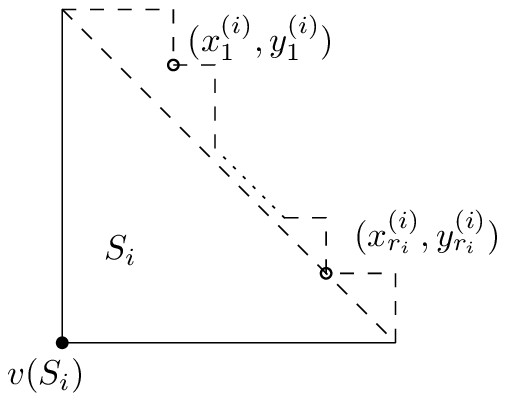}
   \caption{$S_i$}\label{Si}
\end{figure}

\begin{lem}\label{x_coordinate_on_stair_point}
For every $(x',y')\in Z(S_i)$, there exists $j\in\{1,\ldots,N\}\setminus\{i\}$ such that $(x',y')\in S_j$ and $x'=x_0^{(j)}$ where $x_0^{(j)}$ is the $x$-coordinate of $v(S_j)$ (Fig. \ref{SiSj}).
\end{lem}
\begin{proof}
By the definitions of $S_i$ and $Z(S_i)$, it is not hard to see that there exist $T_{i_1},\ldots,T_{i_k}\in\mathcal{C}_i$ such that $I^2(T_{i_1})\cap\cdots\cap I^2(T_{i_k})\neq\emptyset$ and $v(I^2(T_{i_1})\cap\cdots\cap I^2(T_{i_k}))=(x',y')$. This implies that there is a $j\in\{i_1,\ldots,i_k\}$ such that the $x$-coordinate of $v(I^2(T_{j}))$ is $x'$. Clearly, $(x',y')\in T_j\cap I^2(T_j)\subset S_j$ and $v(S_j)=v(I^2(T_{j}))$.
\end{proof}

\begin{figure}[!ht]
  \centering
    \includegraphics[scale=.80]{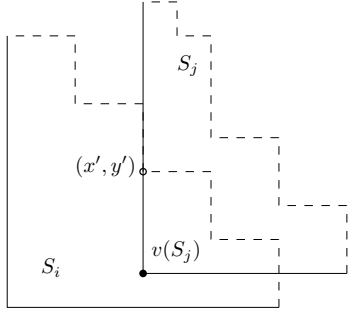}
   \caption{$S_i$ and $S_j$}\label{SiSj}
\end{figure}

\begin{lem}\label{not_exceed_k_1}
Suppose that $i_1,\ldots,i_{k+1}$ are $k+1$ distinct positive integers in $\{1,2,\ldots,N\}$. Then
$$\bigcap_{j=1}^{k+1}S_{i_j}=\emptyset.$$
\end{lem}
\begin{proof}
If $I^2(T_{i_1})\cap\cdots\cap I^2(T_{i_{k+1}})=\emptyset$, then it is obvious that $S_{i_1}\cap\cdots\cap S_{i_{k+1}}= \emptyset$. Assume that $I^2(T_{i_1})\cap\cdots\cap I^2(T_{i_{k+1}})\neq \emptyset$. By Lemma \ref{cut_by_the_others}, we may assume, without loss of generality, that $T_{i_j}$ presses $T_{i_{k+1}}$ for all $j=1,\ldots,k$. Therefore $I^2(T_{i_1})\cap\cdots\cap I^2(T_{i_{k}})\subset R(v(I^2(T_{i_1})\cap\cdots\cap I^2(T_{i_k})))\subset U_{i_{k+1}}$. Hence
$$\bigcap_{j=1}^{k+1}S_{i_j}=(I^2(T_{i_1})\cap\cdots\cap I^2(T_{i_{k+1}}))\setminus(U_{i_1}\cup\cdots\cup U_{i_{k+1}})= \emptyset.$$
\end{proof}

\begin{lem}\label{S_i_exact_k_fold}
$\{S_1,S_2,\ldots,S_N\}$ is a $k$-fold packing of $lI^2$.
\end{lem}
\begin{proof}
Since $T_i\subset lI^2$, it is obvious that $S_i\subset lI^2$. Hence, the result follows immediately from Lemma \ref{not_exceed_k_1}.
\end{proof}

\begin{lem} \label{to_prove_no_R_cut_by_other}
Let $L_i=(\overline{S_i}\setminus S_i)\cap I^2(T_i)$. If $T_i$ presses $T_j$, then $L_i\cap S_j=\emptyset$.
\end{lem}
\begin{proof}
Assume that there is some point $(x,y)\in L_i\cap S_j$. We have that $(x,y)\in U_i$, and hence there exist $T_{i_1},\ldots,T_{i_k}\in \mathcal{C}_i$ such that $(x,y)\in R(v(I^2(T_{i_1})\cap\cdots\cap I^2(T_{i_k})))$. Let $v(I^2(T_{i_1})\cap\cdots\cap I^2(T_{i_k}))=(x',y')$. It is obvious that $x'\leq x$ and $y'\leq y$. Since $(x,y)\in S_j\subset I^2(T_j)$ and $(x',y')\in I^2(T_i)\setminus Int(T_i)$, one can deduce that $(x',y')\in I^2(T_j)$, and hence $I^2(T_j)\cap I^2(T_{i_s})\neq\emptyset$ for all $s=1,\ldots,k$. For $s=1,\ldots,k$, since $T_i$ presses $T_j$ and $T_{i_s}$ presses $T_i$, we have that $T_{i_s}$ presses $T_j$, i.e., $T_{i_s}\in\mathcal{C}_j$. Therefore $(x,y)\in U_j$, which is a contradiction.
\end{proof}

\begin{figure}[!ht]
  \centering
    \includegraphics[scale=.80]{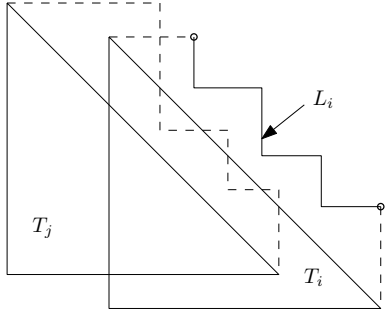}
   \caption{$L_i\cap S_j=\emptyset$}\label{LS}
\end{figure}

\begin{lem}\label{condition_for_cut}
For every $i,j\in\{1,2,\ldots,N\}$, we have $L_i\cap S_j=\emptyset$ or $L_j\cap S_i=\emptyset$.
\end{lem}
\begin{proof}
If $I^2(T_i) \cap I^2(T_j)=\emptyset$ or $i=j$, then the result is trivial. When $I^2(T_i)\cap I^2(T_j)\neq \emptyset$ and $i\neq j$, we have that either $T_i$ presses $T_j$ or $T_j$ presses $T_i$. The result follows directly from Lemma \ref{to_prove_no_R_cut_by_other}.
\end{proof}

\begin{lem}\label{more_than_r_k_1}
For $i=1,2,\ldots,N$, let
$$n_i=card\{S_j : v(S_j)\in Int(S_i)\cup Z(S_i), j=1,\ldots,N\}.$$
Then, we have
$$n_i\geq r_i-k+1.$$
\end{lem}
\begin{proof}
Suppose that $Z(S_i)=\{(x_1^{(i)},y_1^{(i)}),\ldots,(x_{r_i}^{(i)},y_{r_i}^{(i)})\}$. By Lemma \ref{x_coordinate_on_stair_point}, we know that for every $j=1,\ldots,r_i$, there exists an $i_j\in\{1,\ldots,N\}\setminus\{i\}$ such that $(x_j^{(i)},y_j^{(i)})\in S_{i_j}$ and $x_j^{(i)}=x_{i_j}$ where $x_{i_j}$ is the $x$-coordinate of $v(S_{i_j})$ (see Figure \ref{Sij}). Let $y_i$ and $y_{i_j}$ be the $y$-coordinates of $v(S_i)$ and $v(S_{i_j})$, respectively. Let
$$\mathcal{F}=\{S_{i_j} :y_{i_j}\leq y_i, j=1,\ldots,r_i\}.$$
By Lemma \ref{condition_for_cut}, we know that $L_{i_j}\cap S_i=\emptyset$ for all $S_{i_j}\in\mathcal{F}$. We note that $S_i\notin \mathcal{F}$. Since $\{S_1,\ldots,S_N\}$ is a $k$-fold packing of $lI^2$, one can deduce that $card\{\mathcal{F}\}\leq k-1$. It is not hard to see that for every $S\in\{S_{i_1},S_{i_2},\ldots,S_{i_{r_i}}\}\setminus \mathcal{F}$, we have $v(S)\in Int(S_i)\cup Z(S_i)$. Hence
$$n_i\geq card\{Z(S_i)\}-card\{\mathcal{F}\}\geq r_i-k+1.$$
\begin{figure}[!ht]
  \centering
    \includegraphics[scale=.85]{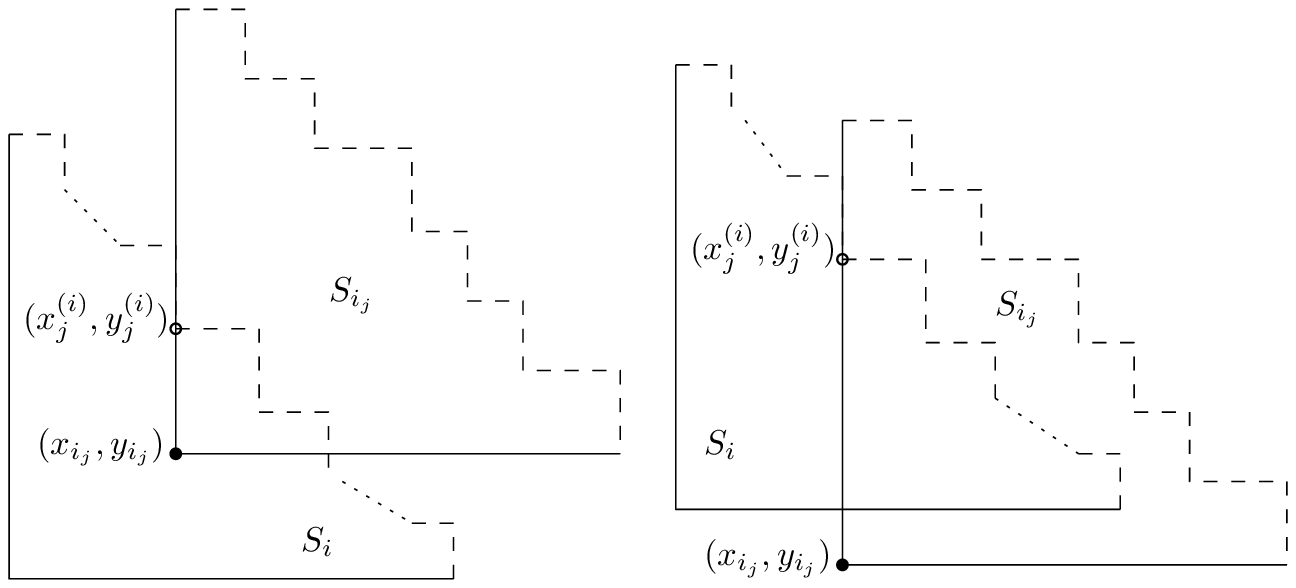}
   \caption{$S_{i_j}$}\label{Sij}
\end{figure}
\end{proof}

\begin{lem}\label{less_than_kN}
$$\sum_{i=1}^{N}n_i\leq kN.$$
\end{lem}
\begin{proof}
For $i=1,\ldots,N$, let $\mathcal{F}_i=\{S_j : v(S_j)\in Int(S_i)\cup Z(S_i), j=1,\ldots,N\}$ and $\mathcal{F}^*_i=\{S_j : v(S_i)\in Int(S_j)\cup Z(S_j), j=1,\ldots,N\}$. Clearly, we have $n_i=card\{\mathcal{F}_i\}$. Let $n_i^*=card\{\mathcal{F}^*_i\}$. It is not hard to show that $\sum_{i=1}^{N}n_i=\sum_{i=1}^{N}n_i^*$. On the other hand,
since $\{S_1,\ldots,S_N\}$ is a $k$-fold packing of $lI^2$, it is obvious that $n_i^*\leq k$. Hence
$$\sum_{i=1}^{N}n_i=\sum_{i=1}^{N}n_i^*\leq kN.$$
\end{proof}
 The following lemma follows immediately from Lemmas \ref{more_than_r_k_1} and \ref{less_than_kN}.

\begin{lem}\label{approximate_average_stair_point}
$$\sum_{i=1}^{N}r_i\leq (2k-1)N.$$
\end{lem}

\section{Proof of Main Theorem}
 Let $\mathcal{T}=\{T_1,T_2,\ldots,T_N\}$ be a \emph{normal} $k$-fold translative packing of $lI^2$ with $T$. Let $S_i$ be the half-open $r_i$-stair polygon defined by $\mathcal{T}$ as shown in Section \ref{key_section}.
By Lemma \ref{S_i_exact_k_fold}, Lemma \ref{approximate_average_stair_point}, the convexity of $A^*$ and (\ref{lattice_density}), one obtains
\begin{align*}
\frac{|T_1|+|T_2|+\cdots+|T_N|}{|lI^2|} &= \frac{N|T|}{|lI^2|}\\
&\leq \frac{kN|T|}{|S_1|+|S_2|+\cdots+|S_N|}\\
&\leq \frac{kN|T|}{A^*(r_1)+A^*(r_2)+\cdots+A^*(r_N)}\\
&\leq \frac{k|T|}{A^*(2k-1)}=\delta_L^k(T),
\end{align*}
and hence
\begin{equation*}
\delta_T^k(T)\leq \delta_L^k(T).
\end{equation*}
This completes the proof.
\section*{Acknowledgment}
This work was supported by 973 Programs 2013CB834201 and 2011CB302401.

\end{document}